\documentclass[leqno,12pt]{amsart} 

\usepackage{amssymb} 
\usepackage{amsmath}
\usepackage{amsthm}
\usepackage{amscd}
\usepackage{bm}
\usepackage{graphicx,psfrag,wrapfig}
\theoremstyle{plain} 
\newtheorem{theorem}{Theorem}[section]  
\newtheorem{lemma}[theorem]{Lemma}

\theoremstyle{definition}

\newtheorem{example}[theorem]{Example}

\newcommand{\affine}{\mathbb{C}}

\newcommand{\const}{\mathrm{const}}
\newcommand{\norm}[1]{\left|\!\left|#1\right|\!\right|}

\newcommand{\vol}{\mathrm{vol}}

\begin{document}

\title[Curvatures of ASD connections over the cylinder]
{Sharp lower bound on the curvatures of ASD connections over the cylinder} 

\author[M. Tsukamoto]{Masaki Tsukamoto}

\date{\today}



\begin{abstract}
We prove a sharp lower bound on the curvatures of non-flat ASD connections over 
the cylinder.
\end{abstract}

\maketitle

\section{Introduction}\label{section: introduction}
The purpose of this note is to calculate explicitly a universal lower bound on the 
curvatures of non-flat ASD connections over the cylinder $\mathbb{R}\times S^3$.

First we fix our conventions.
Let $S^3 = \{x_1^2+x_2^2+x_3^2+x_1^4=1\}\subset \mathbb{R}^4$ be the $3$-sphere equipped with 
the Riemannian metric induced by the Euclidean metric on $\mathbb{R}^4$.
Set $X:=\mathbb{R}\times S^3$.
We give the standard metric on $\mathbb{R}$, and $X$ is equipped with the product metric.

Let $\mathbb{H}$ be the space of quaternions.
Consider $SU(2) = \{x\in \mathbb{H}|\, |x|=1\}$ with the Riemannian metric induced by 
the Euclidean metric on $\mathbb{H}$. (Hence it is isometric to $S^3$ above.)
We naturally identify $su(2) := T_1 SU(2)$ with the imaginary part 
$\mathrm{Im} \mathbb{H} := \mathbb{R} i + \mathbb{R} j + \mathbb{R} k$.
Here $i$, $j$ and $k$ have length $1$.

Let $E := X\times SU(2)$ be the product $SU(2)$-bundle.
Let $A$ be a connection on $E$, and let $F_A$ be its curvature.
$F_A$ is a $su(2)$-valued $2$-form on $X$.
Hence for each point $p\in X$ the curvature $F_A$ can be considered as a linear map
\[ F_{A,p}: \Lambda^2(T_pX)\to su(2).\]
We denote by $|F_{A,p}|_{\mathrm{op}}$ the operator norm of this linear map.
The explicit formula is as follows:
Let $x_1,x_2,x_3,x_4$ be the normal coordinate system on $X$ centered at $p$.
Let $A=\sum_{i=1}^4 A_i dx_i$.
Each $A_i$ is a $su(2)$-valued function.
Then $F(A)_{ij} := F_A(\partial/\partial x_i, \partial/\partial x_j) 
= \partial_i A_j - \partial_j A_i + [A_i,A_j]$.
Since $\partial/\partial x_i \wedge \partial/\partial x_j$ $(1\leq i<j\leq 4)$
become a orthonormal basis of $\Lambda^2(TX)$ at $p$,
the norm $|F_{A,p}|_{\mathrm{op}}$ is equal to  
\[ \sup\left\{\left|\sum_{1\leq i<j\leq 4}a_{ij}F(A)_{ij,p}\right||\, a_{ij}\in \mathbb{R}, 
   \sum_{1\leq i<j\leq 4} a_{ij}^2=1 \right\}. \]
Let $\norm{F_A}_{\mathrm{op}}$ be the supremum of $|F_{A,p}|_{\mathrm{op}}$ over $p\in X$.
The main result is the following.
\begin{theorem}\label{main theorem}
The minimum of $\norm{F_A}_{\mathrm{op}}$ over non-flat ASD connections $A$ on $E$ is 
equal to $1/\sqrt{2}$.
\end{theorem}
The above minimum value $1/\sqrt{2}$ is attained by the following BPST instanton 
(\cite{BPST}).
\begin{example} \label{example: BPS instanton}
We define a $SU(2)$ instanton $A$ on $\mathbb{R}^4$ by 
\[ A := \mathrm{Im}\left(\frac{\bar{x}dx}{1+|x|^2}\right), \quad (x= x_1+x_2i+x_3j+x_4k).\]
By the conformal map 
\[ \mathbb{R}\times S^3\to \mathbb{R}^4\setminus \{0\},\quad 
   (t, \theta) \mapsto e^t \theta ,\]
the connection $A$ is transformed into an ASD connection $A'$ on $E$ over $\mathbb{R}\times S^3$.
Then 
\[ |F_{A',(t,\theta)}|_{\mathrm{op}} = \frac{2\sqrt{2}}{(e^t+e^{-t})^2}.\]
Hence 
\[ \norm{F_{A'}}_{\mathrm{op}} = \frac{1}{\sqrt{2}}.\]
\end{example}

Theorem \ref{main theorem} is a Yang-Mills analogy of the classical result of 
Lehto \cite[Theorem 1]{Lehto} in complex analysis.
(The formulation below is due to Eremenko \cite[Theorem 3.2]{Eremenko}.
See also Lehto-Virtanen \cite[Theorem 1]{Lehto-Virtanen}.) 

Consider $\affine^* := \affine\setminus \{0\}$ with the length element $|dz|/|z|$.
We give a metric on $\affine P^1 = \affine \cup\{\infty\}$ by 
(naturally) identifying it with the unit $2$-sphere $\{x_1^2+x_2^2+x_3^2=1\}$.
For a map $f:\affine^*\to \affine P^1$ we denote its Lipschitz constant by
$\mathrm{Lip}(f)$.

Then Lehto \cite[Theorem 1]{Lehto} proved that 
the minimum of $\mathrm{Lip}(f)$ over non-constant holomorphic maps 
$f:\affine^*\to \affine P^1$ is equal to $1$. 
The function $f(z) = z$ attains the minimum.

Eremenko \cite[Section 3]{Eremenko} discussed the relation between this Lehto's result
and a quantitative homotopy argument of Gromov \cite[Chapter 2, 2.12. Proposition]{Gromov}.
Our proof of Theorem \ref{main theorem} is inspired by this idea.

\section{Preliminaries: Connections over $S^3$} \label{section: preliminaries}

In this section we study the method of choosing good gauges for 
some connections over $S^3$.
The argument below is a careful study of \cite[pp. 146-148]{Freed-Uhlenbeck}.
Set $N:=(1,0,0,0)\in S^3$ and $S:=(-1,0,0,0)\in S^3$.
Let $P:=S^3\times SU(2)$ be the product $SU(2)$-bundle over $S^3$.
For a connection $B$ on $P$ we define the operator norm $\norm{F_B}_{\mathrm{op}}$ 
in the same way as in Section \ref{section: introduction}.

Let $v_1,v_2\in T_NS^3$ be two unit tangent vectors at $N$. ($|v_1|=|v_2|=1$.)
Let $\exp_N:T_NS^3\to S^3$ be the exponential map at $N$.
Since $|v_1|=|v_2|=1$, we have $\exp_N(\pi v_1) = \exp_N(\pi v_2) =S$.
We define a loop $l:[0,2\pi]\to S^3$ by 
\begin{equation*}
 l(t) := \begin{cases}
           \exp_N(tv_1) \quad &(0\leq t\leq \pi) \\
           \exp_N((2\pi-t)v_2) \quad &(\pi\leq t\leq 2\pi).
         \end{cases}
\end{equation*}      
\begin{lemma} \label{lemma: holonomy and curvature}
Let $B$ be a connection on $P$.
Let $\mathrm{Hol}_l(B)\in SU(2)$ be the holonomy of $B$ along the loop $l$.
Then 
\[ d(\mathrm{Hol}_l(B), 1)\leq 2\pi \norm{F_B}_{\mathrm{op}}.\]
Here $d(\cdot,\cdot)$ is the distance on $SU(2)$ defined by the Riemannian metric.
\end{lemma}
\begin{proof}
This follows from the standard fact that curvature is an infinitesimal holonomy 
\cite[p. 36]{Donaldson-Kronheimer}.
($2\pi$ is half the area of the unit $2$-sphere.)
The explicit proof is as follows:
Take a unit tangent vector $v_3\in T_NS^3$ orthogonal to $v_1$ such that there is 
$\alpha\in [0,\pi]$ satisfying 
$v_2= v_1\cos \alpha + v_3\sin \alpha$.
Consider (the spherical polar coordinate of the totally geodesic $S^2\subset S^3$
tangent to $v_1$ and $v_3$):
\[\Phi: [0,\alpha]\times [0,\pi]\to S^3, \quad 
(\theta_1, \theta_2)\mapsto \exp_N\{\theta_2(v_1\cos\theta_1 + v_3\sin\theta_1)\}.\]

Let $Q$ be the pull-back of the bundle $P$ by $\Phi$.
Since $\Phi([0,\alpha]\times \{0\})=\{N\}$ and 
$\Phi([0,\alpha]\times \{\pi\})=\{S\}$, 
$Q$ admits a trivialization under which the pull-back connection 
$\Phi^*B$ is expressed as 
$\Phi^*B=B_1d\theta_1+B_2\theta_2$ with 
$B_1=0$ on $[0,\alpha]\times \{0,\pi\}$.

We take a smooth map $g:[0,\alpha]\times [0,\pi]\to SU(2)$ satisfying 
\[ g(\theta_1,0)= 1 \quad (\forall \theta_1\in [0,\alpha]),\quad 
   (\partial_2+B_2) g=0.\]
We have $\mathrm{Hol}_l(B) = g(\alpha,\pi)^{-1}g(0,\pi)$.
Then $F_{\Phi^*B}(\partial_1, \partial_2)g = 
[\partial_1+B_1,\partial_2+B_2]g= 
-(\partial_2+B_2)(\partial_1+B_1)g$.
Since $B_1=0$ on $[0,\alpha]\times \{0,\pi\}$,
\[ |\partial_1g(\theta_1, \pi)|\leq 
   \int_{\{\theta_1\}\times [0,\pi]}
   |F_{\Phi^*B}(\partial_1,\partial_2)|d\theta_2.\]
Then
\[ d(\mathrm{Hol}_l(B),1)= d(g(0,\pi),g(\alpha,\pi)) \leq 
   \int_{[0,\alpha]\times [0,\pi]}|F_{\Phi^*B}(\partial_1, \partial_2)|
   d\theta_1 d\theta_2.\]
$F_{\Phi^*B}(\partial_1,\partial_2) = 
F_B(d\Phi(\partial/\partial\theta_1),d\Phi(\partial/\partial\theta_2))$.
The vectors $d\Phi(\partial/\partial\theta_1)$ and $d\Phi(\partial/\partial\theta_2)$ are orthogonal 
to each other, and 
$|d\Phi(\partial/\partial\theta_1)| = \sin\theta_2$ and $|d\Phi(\partial/\partial\theta_2)|=1$.
Hence $|F_{\Phi^*B}(\partial_1,\partial_2)|\leq 
\norm{F_B}_{\mathrm{op}} \sin\theta_2$.
Thus, from $0\leq \alpha\leq \pi$,
\[ d(\mathrm{Hol}_l(B),1)\leq 
 \norm{F_B}_{\mathrm{op}}\int_{[0,\alpha]\times [0,\pi]}\sin\theta_2 \, d\theta_1 d\theta_2
 = 2\alpha \norm{F_B}_{\mathrm{op}}\leq 2\pi \norm{F_B}_{\mathrm{op}}.\]
\end{proof}
Let $\tau<1/2$.
Let $B$ be a connection on $P$ satisfying $\norm{F_B}_{\mathrm{op}}\leq \tau$.
We will construct a good connection matrix of $B$.

Fix $v\in T_NS^3$.
By the parallel translation along the geodesic $\exp_N (tv)$ $(0\leq t\leq \pi)$
we identify the fiber $P_S$ with the fiber $P_N$. 
Let $g_N$ and $g_S$ be the exponential gauges
(see \cite[p. 146]{Freed-Uhlenbeck} or \cite[p. 54]{Donaldson-Kronheimer})
centered at $N$ and $S$ respectively:
\[ g_N:P|_{S^3\setminus \{S\}}\to (S^3\setminus \{S\})\times P_N, \quad 
   g_S:P|_{S^3\setminus \{N\}}\to (S^3\setminus \{N\})\times P_N.\]
(In the definition of $g_S$ we identify $P_S$ with $P_N$ as in the above.)
By Lemma \ref{lemma: holonomy and curvature}, for $x\in S^3\setminus \{N,S\}$,
\[ d(g_N(x),g_S(x))\leq 2\pi \norm{F_B}_{\mathrm{op}} \leq 2\pi\tau < \pi.\]
The injectivity radius of $SU(2) = S^3$ is $\pi$ (this is a crucial point of the argument).
Hence there uniquely exists $u(x)\in \mathrm{ad} P_N (\cong su(2))$ satisfying 
\[ |u(x)|\leq 2\pi \norm{F_B}_{\mathrm{op}}, \quad g_S(x) = e^{u(x)}g_N(x).\]
We take and fix a cut-off function $\varphi:S^3\to [0,1]$ such that 
$\varphi(x_1,x_2,x_3,x_4)$ is equal to $0$ over $\{x_1>1/2\}$ and equal to 
$1$ over $\{x_1<-1/2\}$.
We can define a bundle trivialization $g$ of $P$ all over $S^3$ by 
$g := e^{\varphi u}g_N$
Then the connection matrix $g(B)$ satisfies 
\[ |g(B)|\leq C_\tau \norm{F_B}_{\mathrm{op}}.\]   
Here $C_\tau$ is a positive constant depending on $\tau$.

\section{Proof of Theorem \ref{main theorem}} \label{section: proof of main theorem}

In this section we denote by $t$ the standard coordinate of $\mathbb{R}$.
Let $A$ be an ASD connection on $E$ satisfying 
$\norm{F_A}_{\mathrm{op}} < 1/\sqrt{2}$.
We will prove that $A$ must be flat.
Set $\tau:=\norm{F_A}_{\mathrm{op}}/\sqrt{2} <1/2$.

The ASD equation implies that $F_A$ has the following form: 
\[ F_A = -dt\wedge (*_3F(A|_{\{t\}\times S^3})) + F(A|_{\{t\}\times S^3}), \]
where $A|_{\{t\}\times S^3}$ is the restriction of $A$ to $\{t\}\times S^3$ 
and $*_3$ is the Hodge star on $\{t\}\times S^3$.
Hence 
\[ |F_{A,(t,\theta)}|_{\mathrm{op}} = \sqrt{2}|F(A|_{\{t\}\times S^3})_\theta|_{\mathrm{op}}.\]
Therefore 
\[ \norm{F(A|_{\{t\}\times S^3})}_{\mathrm{op}} \leq \tau < \frac{1}{2}\quad (\forall t\in \mathbb{R}).\]
Thus we can apply the construction of Section \ref{section: preliminaries} 
to $A|_{\{t\}\times S^3}$.

Fix a bundle trivialization of $E$ over $\mathbb{R}\times \{N\}$.
(Any choice will do.)
Then the construction in Section \ref{section: preliminaries} gives
a bundle trivialization $g$ of $E$ over $X$ satisfying 
\[ |g(A)|_{\{t\}\times S^3}|\leq C_\tau \norm{F(A|_{\{t\}\times S^3})}_{\mathrm{op}} \quad 
(\forall t\in \mathbb{R}).\]
Set $A' := g(A)$. We consider the Chern-Simons functional 
\[ cs(A') := \mathrm{tr}(A'\wedge F_{A'} - \frac{1}{3}A'^3).\]
For $R>0$
\begin{equation} \label{eq: Stokes theorem}
  \begin{split}
   \int_{[-R,R]\times S^3}|F_A|^2d\vol &= \int_{\{R\}\times S^3}cs(A') 
   -\int_{\{-R\}\times S^3} cs(A') \\
   &\leq \const_\tau \left(\norm{F(A|_{\{R\}\times S^3})}_{\mathrm{op}}
   + \norm{F(A|_{\{-R\}\times S^3})}_{\mathrm{op}}\right).
  \end{split}
\end{equation}
$\norm{F(A|_{\{\pm R\}\times S^3})}_{\mathrm{op}}$ are bounded as $R\to +\infty$.
Thus 
\[ \int_X |F_A|^2 d\vol < +\infty.\]
This implies that the curvature $F_A$ has an exponential decay at the ends
(see \cite[Theorem 4.2]{Donaldson}).
In particular 
\[ \norm{F(A|_{\{\pm R\}\times S^3})}_{\mathrm{op}} \to 0\quad (R\to +\infty).\]
By the above (\ref{eq: Stokes theorem}) 
\[ \int_X |F_A|^2 d\vol = 0.\]
This shows $F_A \equiv 0$. So $A$ is flat.

%%%%%%%%%%%%%%%%%%%%%%%%%%%%%%%%%%%%%%%%%%%%%%%%%%

\vspace{10mm}

\address{ Masaki Tsukamoto \endgraf
Department of Mathematics,
Kyoto University, Kyoto 606-8502, Japan}

\textit{E-mail address}: \texttt{tukamoto@math.kyoto-u.ac.jp}


\begin{thebibliography}{100}




\bibitem{BPST}
A. A. Belavin, A. M. Polyakov, A. S. Schwartz, Y. S. Tyupkin, 
Pseudo-particle solutions of the Yang-Mills equations, 
Phys. Lett., \textbf{59B} (1975) 85-87




\bibitem{Donaldson}
S.K. Donaldson,
 Floer homology groups in Yang-Mills theory,
 with the assistance of M. Furuta and D. Kotschick,
 Cambridge University Press, Cambridge (2002)






\bibitem{Donaldson-Kronheimer}
S.K. Donaldson, P.B. Kronheimer,
 The geometry of four-manifolds,
 Oxford University Press, New York (1990)




\bibitem{Eremenko}
A. Eremenko,
Normal holomorphic curves from parabolic regions to projective spaces,
preprint, Purdue university (1998), arXiv: 0710.1281







\bibitem{Freed-Uhlenbeck}
D.S. Freed, K.K. Uhlenbeck,
 Instantons and four-manifolds, 
 Second edition, Springer-Verlag, New York (1991)




\bibitem{Gromov}
M. Gromov,
Metric structures for Riemannian and non-Riemannian spaces,
based on the 1981 French original,
with appendices by M. Katz, P. Pansu and S. Semmes,
translated from the French by Sean Michael Bates,
Birkh\"{a}user, Boston (1999)



\bibitem{Lehto}
O. Lehto,
The spherical derivative of meromorphic functions 
in the neighborhood of an isolated singularity,
Comment. Math. Helv. \textbf{33} (1959) 196-205



\bibitem{Lehto-Virtanen}
O. Lehto, K.I. Virtanen,
On the behavior of meromorphic functions in the 
neighborhood of an isolated singularity,
Ann. Acad. Sci. Fenn. \textbf{240} (1957) no. 240



 
\end{thebibliography}
\end{document}